\documentclass[letterpaper, 10 pt, conference]{ieeeconf}  % Comment this line out
\IEEEoverridecommandlockouts                              % This command is only
\usepackage{amsthm}
\usepackage{balance}
\newtheorem{rem}{\textbf{Remark}}
\newtheorem{lemma}{\textbf{Lemma}}
\newtheorem{prop}{\textbf{Proposition}}

\newtheorem{thm}{\textbf{Theorem}}

\newtheorem{ass}{\textbf{Assumption}}

\newcommand{\sat}{\mathbf{Sat}}
\usepackage{amsmath,amssymb}
\usepackage{graphicx}
\usepackage{amsfonts}
\usepackage{float}
\usepackage{algorithmicx}
\usepackage{algorithm}
\usepackage{algpseudocode}
\usepackage{lipsum}
\usepackage{subfig}
\usepackage{color}
\usepackage{cite}
\usepackage{epstopdf}
\usepackage{amsthm}
\usepackage{calrsfs}
\usepackage{dsfont}
\DeclareMathAlphabet{\pazocal}{OMS}{zplm}{m}{n}
\algdef{SE}[DOWHILE]{Do}{doWhile}{\algorithmicdo}[1]{\algorithmicwhile\ #1}%
\title{\LARGE \bf
Optimal Safe Controller Synthesis: A Density Function Approach
}

\author{Yuxiao Chen, Mohamadreza Ahmadi, and Aaron D. Ames% <-this % stops a space
\thanks{The authors are with the Department of Mechanical and Civil Engineering, California Institute of Technology,
         Pasadena, CA, 91106, USA. Emails:
        {\tt\small \{chenyx, mrahmadi, ames\}@caltech.edu}}
}

\begin{document}

\maketitle
\thispagestyle{empty}
\pagestyle{empty}

%%%%%%%%%%%%%%%%%%%%%%%%%%%%%%%%%%%%%%%%%%%%%%%%%%%%%%%%%%%%%%%%%%%%%%%%%%%%%%%%
\begin{abstract}
This paper considers the synthesis of optimal safe controllers based on density functions. We present an algorithm for robust constrained optimal control synthesis using the duality relationship between the density function and the value function. The density function follows the Liouville equation and is the dual of the value function, which satisfies Bellman's optimality principle. Thanks to density functions, constraints over the distribution of states, such as safety constraints, can be posed straightforwardly in an optimal control problem. The constrained optimal control problem is then solved with a primal-dual algorithm. This formulation is extended to the case with external disturbances, and we show that the robust constrained optimal control can be solved with a modified primal-dual algorithm. We apply this formulation to the problem of finding the optimal safe controller that minimizes the cumulative intervention. An adaptive cruise control (ACC) example is used to demonstrate the efficacy of the proposed, wherein we compare the result of the density function approach with the conventional control barrier function (CBF) method.
\end{abstract}
\section{Introduction}\label{sec:intro}
Safety is one of the fundamental goals of control synthesis. Controller design techniques, such as control barrier function (CBF) methods, have been proved to be powerful tools for guaranteeing safety of dynamical systems and their application spans over robotics\cite{wang2017safety,nguyen20163d,chen2018obstacle,glotfelter2017nonsmooth} and transportation systems~\cite{nilsson2014preliminary,chen2018validating}. One of the strengths of the CBF is its ability to work with the legacy controller, e.g. a tracking controller,  in a plug-and-play fashion. Given any legacy controller, the CBF acts as the supervisory controller, filtering the input from the legacy controller with minimum intervention necessary to guarantee safety. This supervisory control structure is shown in Fig.\ref{fig:CBF}.

The computation of CBFs typically centers around invariance conditions, which can be enforced via polytopic projection\cite{nilsson2014preliminary}, robust optimization\cite{chen2018data}, and sum of squares constraints \cite{xu2015robustness}. However, one caveat of the CBF method is that it operates myopically, i.e., the CBF is a function of only the current state, and the intervention only depends on the current situation. Although the intervention at every time instance is minimized, the cumulative intervention is not necessarily minimized. If a CBF controller is designed too conservatively, it may use unnecessary intervention when the situation is not dangerous; if a CBF controller is too optimistic, it may allow the state to get too close to the danger set and have to invoke large intervention to prevent the state from entering the danger set. Given a legacy controller, it is not clear how to synthesize the optimal safe controller that minimizes the cumulative intervention. Besides, the computation of the CBF is nontrivial, almost all numerical methods suffer from different levels of conservatism, which further compromises the performance.
\begin{figure}[t]
  \centering
  \includegraphics[width=0.8\columnwidth]{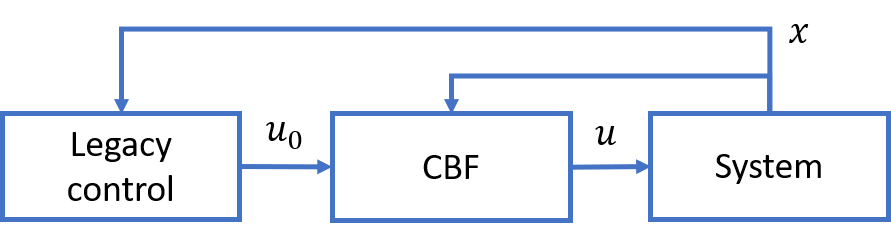}
  \caption{CBF as the supervisory controller}\label{fig:CBF}
\end{figure}

In terms of optimality, optimal control is one of the most well-studied problems in control. Bellman's principle of optimality \cite{bellman2013dynamic} and Pontryagin's maximum principle \cite{pontryagin2018mathematical} are two fundamental theories that solve optimal control problems. With Hamilton-Jacobi-Bellman (HJB) partial differential equation (PDE), we can even solve for the optimal control strategy for the whole state space \cite{bardi2008optimal}. However, when there are safety constraints, such as those requiring that the state should never leave a set or enter a set, the HJB PDE cannot encode those constraints in a clear way. The constrained optimal control problem can be solved for a given initial condition with Lagrangian multipliers \cite{bergounioux1997augemented}, but it only gives solution to a single initial condition rather than the whole state space, and computing the solution online is typically not feasible due to the complexity.

Density functions, proposed by Rantzer in \cite{rantzer2001dual}, are the dual of Lyapunov functions and can be used to verify the stability of nonlinear systems~\cite{prajna2004nonlinear} as well as reachability analysis of polynomial systems~\cite{prajna2007convex}. One related concept is the occupation measure studied in \cite{lasserre2008nonlinear,zhao2019optimal,korda2014convex,majumdar2014convex}, which considers the dual relationship between functions and measures and  solve the optimal control problem with moment programming. We recently showed in~\cite{chen2019duality} that the density function is the dual of the value function in optimal control and one can enforce safety constraints on the density function and solve the constrained optimal control problem with a primal-dual algorithm. In this paper, we take advantage of this duality relationship to design controllers that are both safe and optimal. Furthermore, we consider the case where the dynamical system is subject to exogenous disturbances and propose a synthesis procedure for controllers with safety and optimality. We elucidate our proposed methodology with an adaptive cruise control example.

The rest of the paper is organized as follows. In the next section, we review some preliminary notions and definitions used in the paper. In Section~\ref{sec:dual}, we discuss the duality between the density functions and the optimal control value functions. In Section~\ref{sec:robust_density}, we propose a technique to synthesize controllers that guarantee safety and optimality. In Section~\ref{sec:result}, we elucidate the efficacy of the proposed methodology with an adaptive cruise control example. Finally, in Section~\ref{sec:conclusion}, we conclude the paper.

%Based on these density functions, we are able to find the optimal safe controller in the sense that the cumulative intervention is minimized, and the controller guarantees safety for all possible disturbances.

%{\color{red}[WE NEED A PARAGRAPH THAT CLEARLY OUTLINES OUR CONTRIBUTIONS!]}

%{\color{red}[OUTLINE PARAGRAPH?]}

\textit{Nomenclature}  $\mathbb{N}$ denotes the set of natural numbers, $\mathbb{N}_+$ denotes the positive natural number, $\mathbb{R}$ denotes the set of real numbers. Given a differential equation $\dot{x}=f(x)$, where $f:\pazocal{X}\to \mathbb{R}^n$ is a locally Lipschitz function, $\Phi_f(x_0,T)$ denotes the flow map of the dynamics with initial state $x_0$ and horizon $T$. ${\left\langle {a,b} \right\rangle _\pazocal{X}}=\int_\pazocal{X} {a(x) \cdot b(x)dx} $ denotes the inner product of two functions $a$ and $b$. $\mathbf{0}$ denotes a vector of all zeros or a function that is always zero, depending on the context. $\mathds{1}_S$ denotes the indicator function of a set $S$. We use bold font $\mathbf{u}$ to denote a controller that maps state to the control input, and normal font $u$ to denote the actual control input. For a variable $x\in\pazocal{X}$, $x[\cdot]$ denotes its trajectory over time and $\pazocal{X}[\cdot]$ denote the set of possible trajectories.

%\section{Preliminaries}

\section{Preliminaries}\label{sec:density}
%{\color{red} PUT ALL PRELIMINARIES IN ONE SECTION}

In this section, we review several notion and definitions used throughout the paper.
\subsection{Density functions for dynamical systems}\label{sec:density_intro}

%{\color{red}INITIAL CONDITIONS?}
Density function can be understood as the measure of state concentration in the state space. Given a dynamical system
\begin{equation}\label{dyn_no_control}
  \dot{x}=f(x),\;x\in\pazocal{X}\subseteq\mathbb{R}^n,
\end{equation}
the density function $\rho:[0,\infty)\times\pazocal{X}\to\mathbb{R}$ describes the concentration of states, and its evolution follows the Liouville PDE:
\begin{equation}\label{eq:liouville}
  \begin{aligned}
\frac{{\partial \rho }}{{\partial t}} + \nabla  \cdot \left( {\rho\cdot f} \right) &= \phi(t,x,\rho), \\
\rho \left( {0,x} \right) &= {\rho _0}\left( x \right),
\end{aligned}
\end{equation}
where $\phi:[0,\infty)\times\pazocal{X}\times\mathbb{R}\to\mathbb{R}$ is the supply function, $\phi(t,x_0,\rho(x_0,t))>0$ denotes a source representing where the new states appear, and $\phi(t,x_0,\rho(x_0,t))<0$ denotes a sink, indicating that some states exit the system at $x_0$, time $t$. We let $\phi$ depend on $\rho$ to allow more flexible characterization of the supply.

The Liouville PDE can be transformed and solved as an ordinary differential equation (ODE) since
\begin{equation}\label{eq:Liouville_ODE_prep}
  \frac{{\partial \rho }}{{\partial t}} + \nabla  \cdot \left( {\rho \cdot f} \right) = {\left. {\frac{{d\rho }}{{dt}}} \right|_{\dot x = f(x)}}+(\nabla\cdot f) \rho = \phi.
\end{equation}
This implies that we can integrate the following ODE to get the density function alone the trajectory of the dynamic system $\dot{x}=F(x)$ as
\begin{equation}\label{eq:Liouville_ODE}
  \left[ {\begin{array}{*{20}{c}}
{\dot x}\\
{\dot \rho }
\end{array}} \right] = \left[ {\begin{array}{*{20}{c}}
{f\left( {x} \right)}\\
{\phi \left( {t,x,\rho} \right) - \nabla  \cdot f\left( {t,x} \right)\rho }
\end{array}} \right].
\end{equation}

Then, given an initial density distribution $\rho(0,\cdot)=\rho_0$ and supply $\phi$, the density at state $x_T$, time $T$ can be computed with the following two step process:
\begin{itemize}
  \item First, solve the reverse ODE of $\dot{x}=-f(x)$ with initial condition $x_T$ to get $\Phi_{-f}(x,T)=\Phi_f(x,-T)$.
  \item Then, solve the extended ODE in \eqref{eq:Liouville_ODE} with initial condition $[\Phi_f(x,-T),\rho_0(\Phi_f(x,-T))]^\intercal$ to time $T$.
\end{itemize}
For more detail on the procedure of computing the density function, see~\cite{chen2019duality}.

If $\phi$ does not depend on time and as $t\to\infty$, $\frac{\partial \rho}{\partial t}=0$ everywhere, we say that $\rho$ reaches a stationary distribution, and we denote the stationary density $\rho_s:\pazocal{X}\to\mathbb{R}$, which satisfies
 \begin{equation}\label{eq:rho_s}
   \phi \left(x, \rho_s  \right) - \nabla  \cdot \left( {{\rho_s } \cdot f} \right) = 0.
 \end{equation}
\begin{rem}
  The existence and uniqueness of $\rho_s$ can be guaranteed for certain supply functions. For the case discussed in Section \ref{sec:robust_density}, they can be guaranteed if the discount factor $\kappa$ is large enough. Due to the space limit, we omit the analysis.
\end{rem}

\subsection{Control Barrier Functions}
CBF is a popular and powerful tool to guarantee safety for a dynamical system. While there are several forms of CBFs, we take the zeroing barrier introduced in \cite{ames2017control} as an example. Consider the following dynamical system with disturbances
\begin{equation}\label{eq:dyn_robust}
  \dot{x}=F(x,u,d),
\end{equation}
where $x\in\pazocal{X}$, $u\in\pazocal{U}$ and $d\in\pazocal{D}$ are the state, input and disturbance respectively. $d$ can be measured or unmeasured. Suppose there exists a function $b:\pazocal{X}\to\mathbb{R}$ that satisfies
\begin{equation}\label{eq:CBF}
  \begin{array}{*{20}{l}}
\forall x \in {\pazocal{X}_s},&b(x) \ge 0\\
\forall x \in {\pazocal{X}_d},&b(x) < 0\\
\forall x \in \left\{ {x\mid b(x) \ge 0} \right\},&\\
\exists u \in \pazocal{U}\;\mathrm{s.t.}\;\forall d \in \pazocal{D},&
\dot b + \alpha \left( b \right) \ge 0,
\end{array}
\end{equation}
where $\pazocal{X}_s$ is the safe set and $\pazocal{X}_d$ is the danger set where we want to keep the state away from. $\alpha(\cdot)$ is a class-$\pazocal{K}$ function, i.e., $\alpha(\cdot)$ is strictly increasing and satisfies $\alpha(0)=0$. Then, we can use the following optimization to design a controller that keeps the system safe
\begin{equation}\label{eq:CBF_QP}
  \begin{aligned}
u = \mathop {\arg \min }\limits_{u \in \pazocal{U}} &\left\| {u - {u_0}} \right\|\\
\mathrm{s.t.}&\nabla b\cdot F\left( {x,u,d} \right) + \alpha \left( b \right) \ge 0,
\end{aligned}
\end{equation}
where $u_0$ is the input of the legacy controller. It can be proved that the controller in \eqref{eq:CBF_QP} would render the set $\left\{ {x\mid b(x) \ge 0} \right\}$ invariant and keep any state that starts within $\pazocal{X}_s$ from entering $\pazocal{X}_d$.
% \subsection{HJB PDE for Safe Control Synthesis}
% Another important technique for safe control synthesis is the Hamilton Jacobi Isaac (HJI) PDE approach \cite{mitchell2005time}. Here, HJI is used to solve a differential game between the control and the disturbance. In particular, the following differential game is considered
% \begin{equation}\label{eq:HJB}
%   \begin{aligned}
% V\left( x \right) = \mathop {\max }\limits_{u[ \cdot ] \in \pazocal{U}[ \cdot ]} \mathop {\min }\limits_{d[ \cdot ] \in \pazocal{D}[ \cdot ]}  &\int_0^T {0dt}  + D\left( {x\left( T \right)} \right),\\
% s.t.&\quad \dot x = F\left( {x,u,d} \right),
% \end{aligned}
% \end{equation}
% where $D$ is set up so that $\pazocal{X}_d=\left\{x\mid D(x)\le 0\right\}$. Then $V(x)> 0$ implies that the control wins the game and there exists a control strategy that prevents the state from entering $\pazocal{X}_d$; otherwise, the disturbance wins the game and there exists a disturbance strategy that drives the state to $\pazocal{X}_d$. The differential game is then solved by the following HJI PDE
% \begin{equation}\label{eq:HJB_pde}
%   \frac{{\partial V}}{{\partial t}} + \mathop {\min }\limits_{d \in \pazocal{D}} \mathop {\max }\limits_{u \in \pazocal{U}} \nabla V \cdot F\left( {x,u,d} \right) = 0
% \end{equation}
% The value function $V$ is often used as a CBF, and the implementation follows the same QP as shown in \eqref{eq:CBF_QP}.

\section{Duality between density function and value function}\label{sec:dual}
In this section, we show the duality relationship between the density function and the value function. We consider the following infinite horizon discounted cost function
\begin{equation}\label{eq:discount_cost}
\begin{aligned}
  V(x_0) &= \int_0^\infty  {{e^{ - \kappa \tau }}C\left( {x\left( \tau  \right),u\left( \tau  \right)} \right)d\tau },\\
  \mathrm{s.t.}~\dot{x} &=F(x,u),~\quad x(0)=x_0,
\end{aligned}
\end{equation}
where $\kappa>0$ is the discount factor, and  the dynamics depends on the control input. 
\begin{rem}
Disturbance is not allowed here since it would render the density function undetermined. We will later show how to incorporate disturbance in Section \ref{sec:robust_density}.
\end{rem}
% \begin{equation}\label{eq:dyn_control}
%   \dot{x}=F(x,u).
% \end{equation}
 This is an infinite horizon optimal control problem. Suppose a positive supply function $\phi_+$ is given, i.e., initial states emerge with rate $\phi_+(x)$ at $x$, then we want to minimize the overall cost rate, which can be computed with the following optimization problem:
\begin{equation}\label{eq:primal_discount}
\begin{aligned}
J_p^\star=&\left\langle {V,{\phi _ + }} \right\rangle \\
\mathrm{s.t.}\;&C + \nabla V \cdot F_{\mathbf{u}^\star} -\kappa V = 0\\
&{\mathbf{u}^ \star }(x) = \mathop {\arg \min }\limits_{ u\in\pazocal{U}} \;C + \nabla V \cdot F,
\end{aligned}
\end{equation}
where $F_{\mathbf{u}^\star}(x)\doteq F(x,\mathbf{u}^\star(x))$. 
The cost can be written this way since for every initial state entering into the state space at $x$, it induces a cost $V(x)$. The second and third line is simply Bellman's optimality condition. We denote this problem as the primal optimization.

We posit the following assumption, which apply to most applications.
\begin{ass}\label{ass:supply_support}
  It is assumed that $\phi_+$ is nonzero only inside a compact set, and zero everywhere else.
\end{ass}
% \begin{ass}\label{ass:bound_F}
%   It is assumed that $F$ is bounded, i.e., $\forall x\in\pazocal{X}, \forall u\in\pazocal{U}, {\left\| F(x,u) \right\|_2}\le M$, where $\left\|\cdot\right\|_2$ is the Euclidean norm.
% \end{ass}

If the Liouville PDE converges to a stationary density function $\rho_s$, the overall cost rate can also be computed as the inner product of $\rho_s$ and the running cost, and the two values should be equivalent. Therefore, the dual problem in density function is formulated as
\begin{equation}\label{eq:dual_opt_discount_factor}
    \begin{aligned}
J_d^\star=\mathop {\min }\limits_{\rho_s ,\mathbf{u}}& ~~{\left\langle {\rho_s ,C_\mathbf{u}} \right\rangle}_{\pazocal{X}}\;\\
\mathrm{s.t.}\;&\nabla  \cdot \left( {\rho_s \cdot F_\mathbf{u} } \right) = \phi_+-\kappa\rho_s ,\\
   &\forall x\in \pazocal{X},\mathbf{u}(x)\in\pazocal{U} ,\;\rho_s(x)  \ge 0,
\end{aligned}
\end{equation}
where $C_\mathbf{u}(x)\doteq C(x,\mathbf{u}(x))$, and $-\kappa\rho_s$ is the negative supply caused by the discount factor. Before presenting the main result, we need the following additional assumption. 
\begin{ass}\label{ass:differentiable}
  The solutions $\rho_s$ and $V$ to \eqref{eq:primal_discount} and \eqref{eq:dual_opt_discount_factor} are bounded and differentiable.
\end{ass}
We are now ready to present the main result of this paper: 
\begin{thm}\label{thm:duality}
 For a control system described in \eqref{eq:discount_cost}, if $F$ is bounded, i.e., $\forall x\in\pazocal{X}, \forall u\in\pazocal{U}, {\left\| F(x,u) \right\|_2}\le M$, and Assumption \ref{ass:supply_support}, \ref{ass:differentiable} are true, the optimization in \eqref{eq:primal_discount} and \eqref{eq:dual_opt_discount_factor} are dual to each other. If both problems are feasible, there is no duality gap.
\end{thm}
Before proving Theorem \ref{thm:duality}, we need the following lemma:
\begin{lemma}\label{lemma:adjoint_condition}
   For a $\rho_s$ that satisfies \eqref{eq:rho_s} with $\phi(x,\rho_s)=\phi_+(x)-\kappa \rho_s$, $\phi_+$ satisfying Assumption \ref{ass:supply_support}, let $S(R)$ be the sphere with radius $R$ centered around the origin. For a given vector field $\dot{x}=f(x)$ satisfying $\forall x\in\pazocal{X},\left\|f(x)\right\|_2\le M$, define
  \begin{equation}\label{eq:sphere_int}
    g(R) = \int_{S(R)} {{\rho _s}f \cdot \mathord{\buildrel{\lower3pt\hbox{$\scriptscriptstyle\rightharpoonup$}}
\over n} } dS,
  \end{equation}
  then $\mathop {\lim }\limits_{R \to \infty } g\left( R \right) = 0$.
\end{lemma}
\begin{proof}

  Take the derivative of $g$ over $R$:
  \begin{equation}\label{eq:deriv_g}
    \frac{{dg}}{{dR}} = \mathop {\lim }\limits_{\Delta R \to 0} \frac{{\int_{S(R+\Delta R)} {{\rho _s}f \cdot \mathord{\buildrel{\lower3pt\hbox{$\scriptscriptstyle\rightharpoonup$}}
\over n} } dS - \int_{S(R )} {{\rho _s}f \cdot \mathord{\buildrel{\lower3pt\hbox{$\scriptscriptstyle\rightharpoonup$}}
\over n} } dS}}{{\Delta R}}
  \end{equation}
  Note that the numerator is the surface integral of the thin hull between $S(R)$ and $S(R+\Delta R)$, defined as
  \begin{equation}\label{eq:hull}
    H\left( {R,R + \Delta R} \right) \triangleq \left\{ {x\in\mathbb{R}^n\mid R \le \left\| x \right\| \le R + \Delta R} \right\}
  \end{equation}
  Then by the divergence theorem:
  \begin{equation}\label{eq:lemma_div}
    \begin{aligned}
&\int_{S(R + \Delta R)} {{\rho _s}f \cdot \mathord{\buildrel{\lower3pt\hbox{$\scriptscriptstyle\rightharpoonup$}}
\over n} } dS - \int_{S(R)} {{\rho _s}f \cdot \mathord{\buildrel{\lower3pt\hbox{$\scriptscriptstyle\rightharpoonup$}}
\over n} } dS\\
 =& \oint_{\partial H\left( {R,R + \Delta R} \right)} {{\rho _s}f \cdot \mathord{\buildrel{\lower3pt\hbox{$\scriptscriptstyle\rightharpoonup$}}
\over n} } dS\\
 =& \int_{H\left( {R,R + \Delta R} \right)} {\nabla  \cdot \left( {{\rho _s}f} \right)} dx\\
 =& \int_{H\left( {R,R + \Delta R} \right)} {\left( {{\phi _ + } - \kappa {\rho _s}} \right)} dx
\end{aligned}
  \end{equation}
  The argument $x$ is omitted for notational simplicity. By Assumption \ref{ass:supply_support}, there exists a $R_0>0,\forall \left\| x \right\|\ge R_0,\phi_+(x)=0$. By the boundedness of $f$,
  \[g\left( R \right) = \int_{S(R)} {{\rho _s}f \cdot \mathord{\buildrel{\lower3pt\hbox{$\scriptscriptstyle\rightharpoonup$}}
\over n} } dS \le M\int_{S(R)} {{\rho _s}} dS.\]
Therefore
\begin{equation}\label{lemma_proof_conclusion}
\begin{aligned}
\forall R \ge {R_0},\frac{{dg}}{{dR}} &= \mathop {\lim }\limits_{\Delta R \to 0}  - \frac{1}{{\Delta R}}\int_{H\left( {R,R + \Delta R} \right)} {\kappa {\rho _s}} dx\\
 &=-\frac{1}{\Delta R} \kappa \int_{S(R)} {{\rho _s}} dS \cdot \Delta R \le  - \frac{{g\left( R \right)}}{M}
\end{aligned}
\end{equation}
  which indicates that $\mathop {\lim }\limits_{R \to \infty } g\left( R \right) = 0$.
\end{proof}
\begin{proof}[Proof of Theorem \ref{thm:duality}]
We show one direction, from \eqref{eq:dual_opt_discount_factor} to \eqref{eq:primal_discount}, and the other direction is similar. The Lagrangian is formulated as
\begin{equation}\label{eq:lagrangian_main_thm}
 \pazocal{L} = \left\langle {C_\mathbf{u},{\rho _s}} \right\rangle  + \left\langle {\mu ,{\phi _ + } - \kappa {\rho _s} - \nabla  \cdot \left( {{\rho _s}\cdot F_\mathbf{u}} \right)} \right\rangle-\left\langle {\lambda,\rho_s}  \right\rangle,
\end{equation}
where $\mu:\pazocal{X}\to\mathbb{R}$ and $\lambda:\pazocal{X}\to\mathbb{R}_+$ are the Lagrangian multipliers.

% First notice that under Assumption \ref{ass:supply_support} and with the dissipation caused by $\kappa$, it can be easily shown that
% \[\mathop {\lim }\limits_{\left\| x \right\| \to \infty } {\rho _s}\left( x \right) = 0.\]
% Therefore, we can take $\pazocal{X}$ to be large enough and use the adjoint relationship:
By Lemma \ref{lemma:adjoint_condition}, we can use the adjoint relationship:
\[\left\langle {\mu ,\nabla  \cdot \left( {{\rho _s}F_\mathbf{u}} \right)} \right\rangle  =  - \left\langle {\nabla \mu ,{\rho _s}F_\mathbf{u}} \right\rangle  =  - \left\langle {\nabla \mu  \cdot F_\mathbf{u},{\rho _s}} \right\rangle. \]
Then the Lagrangian can be simplified to
\begin{equation}\label{eq:Lagrangian1}
  \pazocal{L}  = \left\langle {C_\mathbf{u} + \nabla \mu  \cdot F_\mathbf{u} - \kappa \mu -\lambda,{\rho _s}} \right\rangle  + \left\langle {\mu ,{\phi _ + }} \right\rangle.
\end{equation}
The Kuhn-Karush-Tucker (KKT) condition reads:
\begin{itemize}
  \item Stationarity condition:
\begin{equation}\label{eq:stationarity}
\begin{aligned}
  \frac{{\partial \pazocal{L}}}{{\partial \rho_s }} &=& {C_\mathbf{u} + \nabla \mu  \cdot F_\mathbf{u} - \kappa \mu } &= &0\\
  \frac{{\partial \pazocal{L}}}{{\partial u }} &=& \frac{{\partial C}}{{\partial u}} + \nabla \mu  \cdot \frac{{\partial F}}{{\partial u}}  &= &0
  \end{aligned}
\end{equation}
  \item Complementary slackness:
\begin{equation}\label{eq:complementary_slackness}
\mu \cdot(\phi-\nabla  \cdot \left( {\rho_s  \cdot F_\mathbf{u}} \right)) =  {\lambda}\cdot\rho_s   =  0
\end{equation}
\end{itemize}
This implies that when $\rho_s>0$, i.e. for area in $\pazocal{X}$ with nonzero density,
\begin{equation}\label{eq:dual_optimality_condition}
  \begin{array}{c}
{\mathbf{u}^ \star(x) } = \mathop {\arg \min }\limits_{u\in\pazocal{U}} \;C(x,u) + \nabla \mu\cdot F(x,u),\\
C_{\mathbf{u}^\star}  +\nabla \mu \cdot F_{\mathbf{u}^\star}-\kappa \mu = 0,
\end{array}
\end{equation}
 which directly comes from the stationarity condition and utilized the fact that $\rho_s>0\to \lambda=0$.

Replacing $\mu$ with $V$, we get the primal optimization in \eqref{eq:primal_discount}. Besides, from \eqref{eq:Lagrangian1}, if such an solution to the optimal problem exists, the dual objective becomes
\begin{equation}
  {J_d^ \star } = \mathop {\max }\limits_{\lambda ,\mu } \mathop {\min }\limits_{{\rho _s},{\bf{u}}} \pazocal{L} = \left\langle {{\phi _ + },\mu } \right\rangle=J_p^ \star ,
\end{equation}
 which shows that there is no duality gap.
\end{proof}

\section{Optimal safe control using \\density functions}\label{sec:robust_density}
In this section, we present the synthesis method for the optimal safe controller and compare the proposed density function based method to some benchmarks.

\subsection{Optimal Safe Control with Density Function Optimization}
We would like to solve the following constrained optimal control problem:
\begin{equation}\label{eq:prob_def}
\begin{aligned}
  \min\; &\int_0^\infty  {{e^{ - \kappa t}}{C(x,u)}dt} \\
  \mathrm{s.t.}\; & \forall x_0\in\pazocal{X}_0,\forall d[\cdot]\in\pazocal{D}[\cdot],\forall t\in[0,\infty),\\
       & \Phi_{F(\cdot,\mathbf{u}(\cdot),d[\cdot])}(x_0,t)\notin\pazocal{X}_d,
  \end{aligned}
\end{equation}
where $\Phi_{F(\cdot,\mathbf{u}(\cdot),d[\cdot])}$ denotes the flow map of the dynamics in \eqref{eq:dyn_robust} under controller $\mathbf{u}$ and disturbance trajectory $d[\cdot]$.

% \begin{rem}
%   We consider this problem instead of the difference between trajectories that follow $\mathbf{u}$ and $\mathbf{u_0}$ because the optimal solution of the latter problem is a controller that depends on the initial condition, which is not well-posed. In addition,  it was shown in~\cite{da2017combining} that given an optimal controller, if the trajectory is perturbed, the optimal strategy is not to return to the original trajectory, but rather to follow the optimal controller at the perturbed state. Under this philosophy, we believe that \eqref{eq:prob_def} is the right problem to solve.
% \end{rem}

First, we solve the safe control synthesis problem with the disturbance as a fixed function of state $d(t)=\mathbf{d}(x(t))$. In this case, the constrained optimal control problem can be stated in the density form as
\begin{equation}\label{eq:density_safety_opt0}
  \begin{aligned}
\mathop {\min }\limits_{\bf{u},{\rho _s}}& \left\langle {{C_\mathbf{u}},{\rho _s}} \right\rangle  \\
\mathrm{s.t.} &\left\langle {\mathds{1}_{\pazocal{X}_d},{\rho _s}} \right\rangle\le 0 \\
& \nabla  \cdot \left( {{\rho _s} \cdot F(x,{\bf{u}}(x),\mathbf{d}(x))} \right) = \phi_+-\kappa\rho_s\\
&\forall x \in {\pazocal{X}},{\bf{u}}(x) \in {\pazocal{U}}.
\end{aligned}
\end{equation}
where $\mathds{1}_{\pazocal{X}_d}$ is the indicator function of the danger set $\pazocal{X}_d$. Take the Lagrangian of \eqref{eq:density_safety_opt0}, comparing to \eqref{eq:lagrangian_main_thm}, an additional term shows up due to the safety constraint, and the Lagrangian becomes
\begin{equation}\label{eq:lagrangian_safe_opt}
\resizebox{.98\hsize}{!}{$
\begin{aligned}
 \pazocal{L} &= \left\langle {C_\mathbf{u}-\lambda,{\rho _s}} \right\rangle  + \left\langle {\mu ,{\phi _ + } - \kappa {\rho _s} - \nabla  \cdot \left( {{\rho _s}F_\mathbf{u,d}} \right)} \right\rangle+\left\langle {\sigma,\rho_s\mathds{1}_{\pazocal{X}_d}}  \right\rangle\\
  &= \left\langle {C_\mathbf{u} + \nabla \mu  \cdot F_\mathbf{u,d} - \kappa \mu -\lambda+\sigma \mathds{1}_{\pazocal{X}_d},{\rho _s}} \right\rangle  + \left\langle {\mu ,{\phi _ + }} \right\rangle,
 \end{aligned}
 $}
\end{equation}
where $F_\mathbf{u,d}(x)\doteq F(x,\mathbf{u}(x),\mathbf{d}(x))$ is the dynamics under $\mathbf{u}$ and $\mathbf{d}$, and $\sigma$ is the dual variable induced by the safety constraint. This shows that the safety constraint adds a perturbation term $\sigma \mathds{1}_{\pazocal{X}_d}$ to the optimality condition for the primal value function optimization, and the primal value function problem becomes
\begin{equation}\label{eq:primal_opt_pert}
  \begin{array}{l}
\nabla V \cdot F_{\mathbf{u}^\star,\mathbf{d}} + C_{\mathbf{u}^\star} + \sigma {\mathds{1}_{{\pazocal{X}_d}}} - \kappa V = 0,\\
{{\bf{u}}^ \star }\left( x \right) = \mathop {\arg \min }\limits_{u \in \pazocal{U}} \nabla V \cdot F\left( {x,u,\mathbf{d}(x)} \right)+C(x,u).
\end{array}
\end{equation}
   This relationship is then used to design a primal-dual algorithm that solves the constrained optimal control problem, as shown in Algorithm \ref{alg:primal_dual_opt_con}.
The algorithm iterates between the primal value function optimization and the density function evaluation. In each iteration, the primal optimal control problem is solved with the current $\sigma$ and gives an optimal control policy $\mathbf{u}^\star$, which is then used to evaluate the density function. Then the perturbation term $\sigma$ is updated based on the density function under $\mathbf{u}^\star$, and the iteration continues until the KKT condition is satisfied up to precision $\epsilon$.
\begin{algorithm}[H]
    \caption{Primal-dual algorithm for optimal control with safety constraint}
    \label{alg:primal_dual_opt_con}
    \begin{algorithmic}[1] % The number tells where the line numbering should start
            \State  $\sigma(0) \gets \mathbf{0}$, $k=0$
            \Do
                \State Solve \eqref{eq:primal_opt_pert} with $\sigma(k)$, get $\mathbf{u}^\star$.
                \State Estimate stationary density $\rho_s$ under $\mathbf{u}^\star$.
                \State ${\sigma (k+1)} \gets \max\left\{\mathbf{0},{\sigma(k)} + \alpha \left( {\rho _s   \mathds{1}_{\pazocal{X}_d}} \right)\right\}$.
                \State $k\gets k+1$
            \doWhile {$\left\| {{\rho _s}\mathds{1}_{\pazocal{X}_d}} \right\|_{\infty} > \epsilon $ }
            \State \textbf{return} $\mathbf{u}^\star,\rho_s,V$
    \end{algorithmic}
\end{algorithm}

We then proceed to solve the robust safe control synthesis problem. Based on the solution for fixed $\mathbf{d}$, the robust density function optimization takes the following form:
%\begin{equation}\label{eq:density_safety_opt}
%  \begin{aligned}
%\mathop {\min }\limits_{\bf{u},{\rho _s}}& \left\langle {{{\left\| {\bf{u} - \bf{u_0}} \right\|}^2},{\rho _s}} \right\rangle \\
%s.t.\;&\nabla  \cdot \left( {{\rho _s} \cdot F(x,{\bf{u}}(x),{\bf{d}}(x))} \right) = {\phi _ + } - \kappa {\rho _s}\\
%&\forall x \in \pazocal{X},{\bf{u}}(x) \in \pazocal{U},\rho_s(x)\ge 0\\
%&\forall {\bf{d}}(x) \in \pazocal{D},\left\langle {{\mathds{1}_{\pazocal{X}_d}},{\rho _s}} \right\rangle  \le 0,
%\end{aligned}
%%\end{equation}
\begin{equation}\label{eq:density_safety_opt1}
\resizebox{.85\hsize}{!}{$
\begin{aligned}
\mathop {\min }\limits_{\bf{u},{\rho _s}}& \left\langle {{C_\mathbf{u}},{\rho _s}} \right\rangle  \\
\mathrm{s.t.} &\left\{\begin{aligned}\mathop {\max }\limits_{\bf{d}} &\left\langle {\mathds{1}_{\pazocal{X}_d},{\rho _s}} \right\rangle  \\
\mathrm{s.t.}& \nabla  \cdot \left( {{\rho _s} \cdot F_\mathbf{u,d}} \right) = \phi_+-\kappa\rho_s,\mathbf{d}(x)\in\pazocal{D}\end{aligned}\right\}\le0,\\
&\forall x \in {\pazocal{X}},{\bf{u}}(x) \in {\pazocal{U}},
\end{aligned}
$}
\end{equation}
The optimization in \eqref{eq:density_safety_opt1} solves for $\bf{u}$ and $\rho_s$ such that under any possible disturbance as a function of state, the stationary density inside the danger set is zero. This is a robust optimization as the constraint should hold for the worst-case $\mathbf{d}$.

Note that the value inside the parentheses in \eqref{eq:density_safety_opt1} is an optimal control problem in the form of density function. From Theorem \ref{thm:duality}, the density function optimization is equivalent to the following optimal control problem:
\begin{equation}\label{eq:worst_d_opt}
  \begin{aligned}
\mathop {\max }\limits_{\mathbf{d}} &\int_0^\infty  {e^{-\kappa t}{\mathds{1}_{{\pazocal{X}_d}}(x)}} dt\\
\mathrm{s.t.}&~~\dot x = F\left( {x,\mathbf{u}(x),\mathbf{d}(x)} \right),
\end{aligned}
\end{equation}
which can be solved with standard HJB PDE.

\begin{prop}
  Under Assumption \ref{ass:memoryless}, the worst-case disturbance signal is a function of $x$.
\end{prop}
\begin{proof}
  By Assumption \ref{ass:memoryless}, the input only depends on the current state $x$ and the dynamics is time invariant. Given a state $x$, the status of the differential game is completely determined by $x$. Let $V^d$ be the value function of the optimal control problem in \eqref{eq:worst_d_opt}, the worst case disturbance input at $x$ is then:
  \begin{equation}\label{eq:d^star_x}
    d^{\star} = \mathop {\arg \max }\limits_{d\in\pazocal{D}} \nabla V^d\cdot F(x,\mathbf{u}^{\star}(x),d),
  \end{equation}
  which is a function of $x$.
\end{proof}

Next, we slightly modify the primal-dual algorithm in Algorithm \ref{alg:primal_dual_opt_con} to solve the robust synthesis problem in \eqref{eq:density_safety_opt1}. Starting with the robust constraint, denote the solution to \eqref{eq:worst_d_opt} as $\mathbf{d}^\star_{\mathbf{u}}$, since it only depends on $\mathbf{u}$. Then, the robust optimization in \eqref{eq:density_safety_opt1} is simplified to
\begin{equation}\label{eq:density_safety_opt2}
  \begin{aligned}
\mathop {\min }\limits_{\bf{u},{\rho _s}}& \left\langle {{C_\mathbf{u}},{\rho _s}} \right\rangle  \\
\mathrm{s.t.} &\left\langle {\mathds{1}_{\pazocal{X}_d},{\rho _s}} \right\rangle\le 0 \\
& \nabla  \cdot \left( {{\rho _s} \cdot F(x,{\bf{u}}(x),{\bf{d}^\star_{\bf{u}}}(x))} \right) = \phi_+-\kappa\rho_s\\
&\forall x \in {\pazocal{X}},{\bf{u}}(x) \in {\pazocal{U}}.
\end{aligned}
\end{equation}
The following primal-dual algorithm solves the robust safe synthesis problem.

\begin{algorithm}[H]
    \caption{Primal-dual algorithm for robust safe control synthesis}
    \label{alg:primal_dual_opt_con_robust}
    \begin{algorithmic}[1] % The number tells where the line numbering should start
            \State  $\sigma(0) \gets \mathbf{0}$, $k=0$
            \Do
                \State Solve \eqref{eq:primal_opt_pert} with $\sigma(k)$, get $\mathbf{u}^\star$.
                \State Solve \eqref{eq:worst_d_opt} with $\mathbf{u}^\star$ to get the worst case $\mathbf{d}^\star$
                \State Estimate stationary density $\rho_s$ under $\mathbf{u}^\star$ and $\mathbf{d}^\star$.
                \State ${\sigma (k+1)} \gets \max\left\{\mathbf{0},{\sigma(k)} + \alpha \left( {\rho _s  \mathds{1}_{\pazocal{X}_d}} \right)\right\}$.
                \State $k\gets k+1$
            \doWhile {$\left\| {\rho_s\mathds{1}_{\pazocal{X}_d}} \right\|_{\infty} > \epsilon $ }
            \State \textbf{return} $\mathbf{u}^\star,\rho,V$

    \end{algorithmic}
\end{algorithm}
The only difference to~Algorithm \ref{alg:primal_dual_opt_con} is the additional step that computes the worst-case disturbance $\mathbf{d}^\star$.

Coming back to the problem of synthesizing the optimal safe controller. For a given legacy controller $\mathbf{u_0}$, the implementation of CBF in \eqref{eq:CBF_QP} is minimizing the intervention of the CBF, but it does not necessarily minimize the cumulative intervention over time. To simplify the problem, we make the following assumption.
\begin{ass}\label{ass:memoryless}
 The legacy controller $\mathbf{u_0}$ is a memoryless state feedback controller.
\end{ass}
Then let
\begin{equation}
    C(x,u)=\left\|u-\mathbf{u_0}(x)\right\|^2,
\end{equation}
which fits into the setup in \eqref{eq:density_safety_opt1} and can be solved with Algorithm \ref{alg:primal_dual_opt_con_robust}.
\subsection{Comparison and discussion}
Similar to the control barrier function, the density function-based safe control synthesis can also guarantee safety robustly under disturbance, but solves a horizon optimization instead of solving myopic optimization at every time step. It is expected to perform better than the CBF, as will be shown in Section \ref{sec:result}. In fact, the result of the robust safe control synthesis in \eqref{eq:density_safety_opt1} should be the optimal safe controller.

Comparing to the finite-horizon HJI approach in \cite{mitchell2005time}, the density approach solves two optimal control problems instead of one. In the differential game setup in \cite{mitchell2005time}, the disturbance and control share the same value function and solves a zero-sum game; while in the density optimization in \eqref{eq:density_safety_opt1}, the control and disturbance optimize different cost functions, and the control strategy has to robustly satisfy a constraint that depends on the disturbance strategy. This separation of cost and constraint allows the method to optimize the performance while guaranteeing safety.

Comparing to the occupation measure approach, the occupation measure depends on the input, and does not explicitly use Bellman's principle of optimality. Therefore, there is no value function defined. Density function can be viewed as the projection of the occupation measure when the input is determined by a certain controller, and we enforce that controller to satisfy Bellman's principle of optimality.

\section{Application to Adaptive Cruise Control}\label{sec:result}
Adaptive Cruise Control (ACC) using CBFs was studied in \cite{ames2014control} and we use this example to demonstrate the proposed density approach. We consider a simple kinetic model 
\begin{equation}\label{eq:ACC_dyn}
  \begin{bmatrix}
    \dot{v_l}, &
    \dot{v}, &
    \dot{D}
    \end{bmatrix}^\intercal=
    \begin{bmatrix}
      a_l, &
      a, &
      v_l-v \end{bmatrix}^\intercal,
\end{equation}
where $v_l$ and $a_l$ are the velocity and acceleration of the lead vehicle, $v$ and $a$ are the velocity and acceleration of the ego vehicle, and $D$ is the distance between the two. We assume
\begin{equation}\label{eq:ACC_bound}
  v,v_l\in[0, v_{\max}],\;\;\;a,a_l\in[-a_{\max}, a_{\max}].
\end{equation}
The safety constraint is given by $D\ge D_{\min}$.

For this simple dynamics and simple constraint, there exists a critical CBF:
\begin{equation}\label{eq:ACC_CBF}
  b\left( x \right) = D - {D_{\min }} - \frac{{{v^2} - v_l^2}}{{2{a_{\max }}}}.
\end{equation}
\begin{prop}
  When $b<0$, there exists a disturbance strategy that results in violation of the safety constraint for all possible control strategy; when $b\ge0$, there exists a control strategy that guarantees safety.
\end{prop}
\begin{proof}
  The optimal control and worst-case disturbance strategies are both taking the minimum acceleration $-a_{\max}$. Simple algebraic calculation proves the proposition.
\end{proof}
Then, the CBF is implemented with the QP shown in \eqref{eq:CBF_QP}. For simplicity, we let the class-$\pazocal{K}$ function to be a linear function $\alpha\cdot b$ with tuning parameter $\alpha$.
\begin{figure*}
   \centering
   \begin{minipage}[t]{5cm}
		\includegraphics[width=5.5cm]{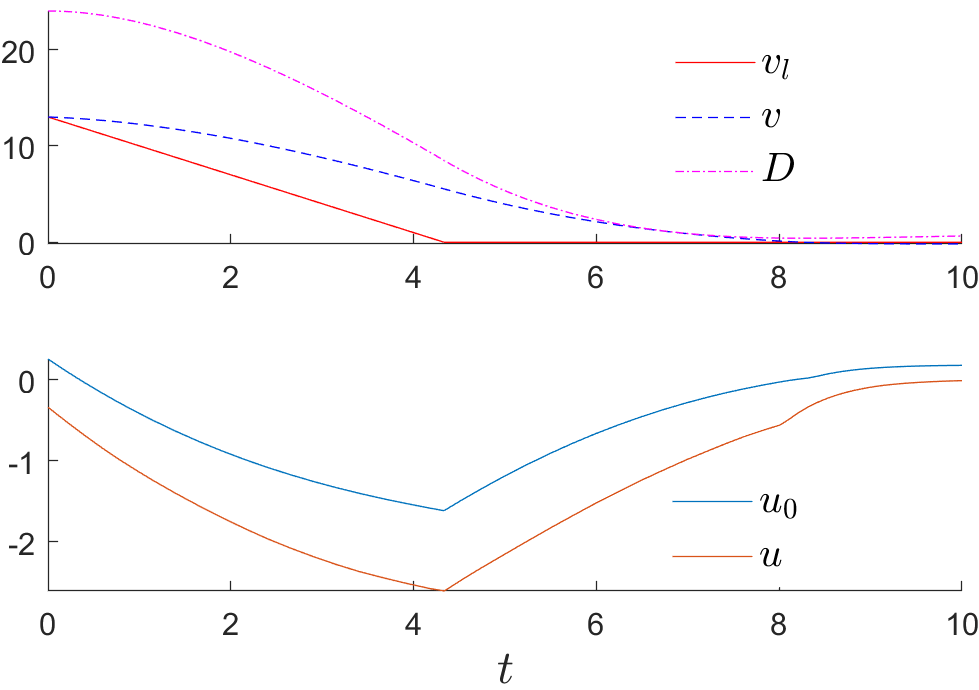}
\center{(a)}
	\end{minipage}
    \begin{minipage}[t]{5cm}
		\includegraphics[width=5.5cm]{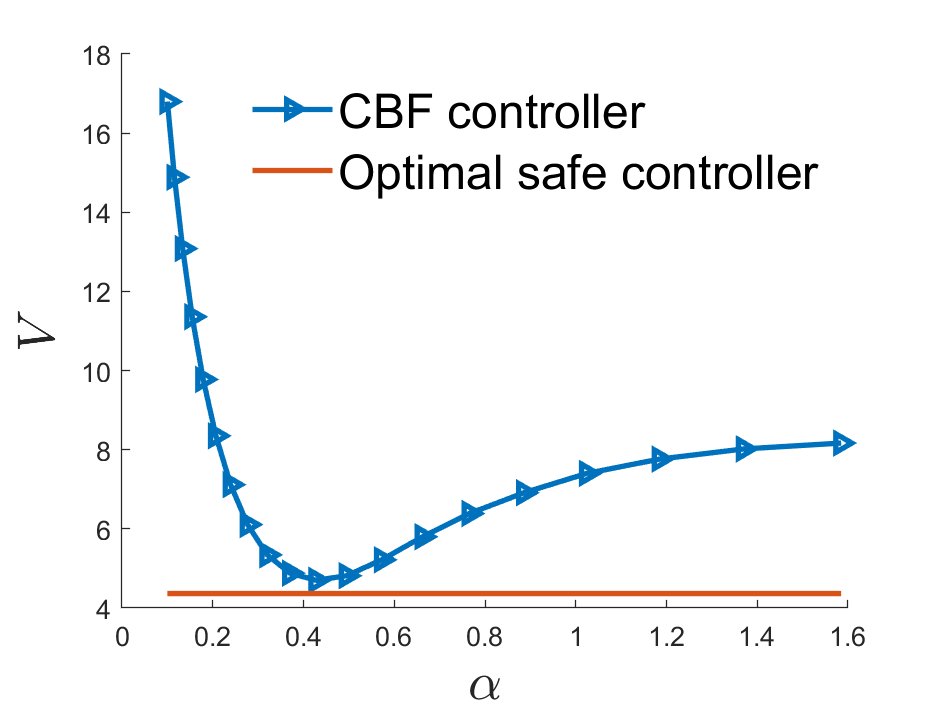}
\center{(b)}
	\end{minipage}
    \begin{minipage}[t]{5cm}
		\includegraphics[width=5.5cm]{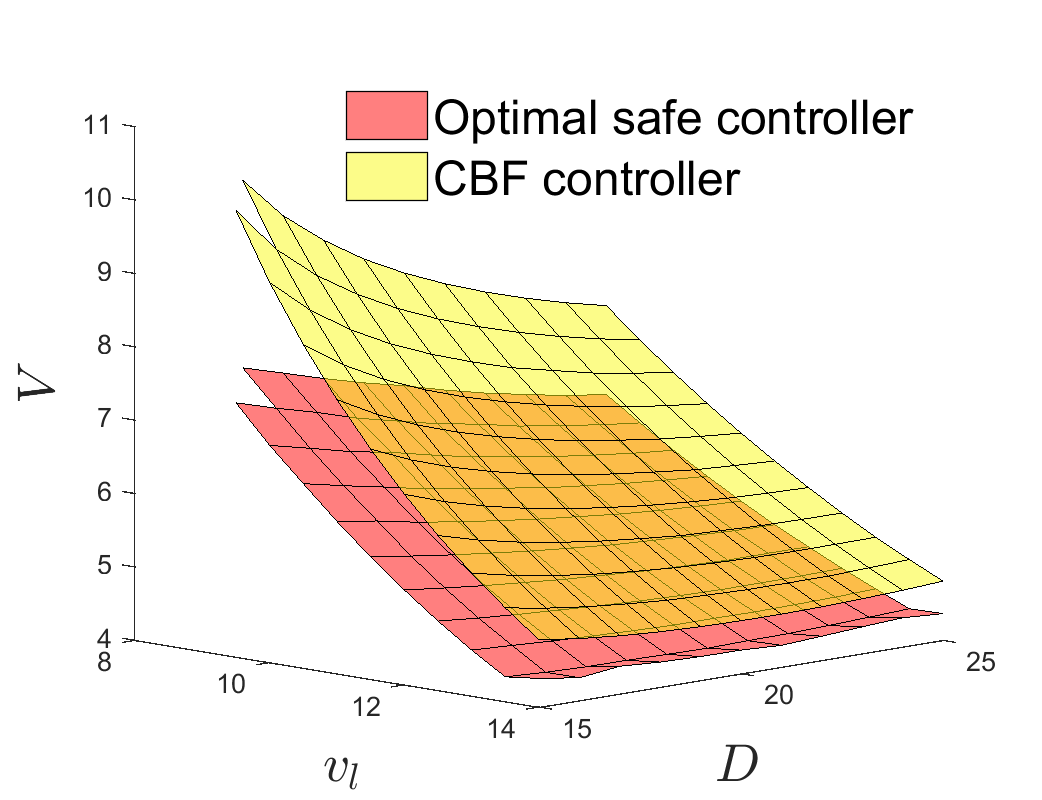}
\center{(c)}
	\end{minipage}
   \caption{Simulation result}\label{fig:CBF_comparison}
   \vspace{-8pt}
 \end{figure*}
The design of $\mathbf{u_0}$ follows a simple LQR approach. The goal is to maintain a desired time headway $\tau_{des}=1.4s$, i.e.
\begin{equation}\label{eq:LQR_cost}
  V = \int \left(\left( D - \tau _{des}v \right)^2 + R{a^2}\right)dt.
\end{equation}
After solving the Riccati equation and obtained the gains $K_v$ and $K_D$, $\mathbf{u_0}$ is defined as
\begin{equation}\label{eq:ACC_u_0}
  \mathbf{u_0}(x)=\sat_{[-a_{\max},a_{\max}]} (K_v(v-v_l)+K_d(D-\tau_{des}v)),
\end{equation}
where $\sat_S(\cdot)$ saturates the signal to keep it within $S$.

With the $\mathbf{u_0}$ given, the robust density optimization is solved with the primal-dual algorithm introduced in Section \ref{sec:robust_density}. The HJB PDE is solved by discretizing the state space and integrating numerically, and the density function is evaluated with the two-step ODE procedure introduced in Section \ref{sec:density_intro}. The resulting controller $\mathbf{u}$ is an array that assigns value to every grid point in the HJB computation and we use linear interpolation to obtain a continuous controller.

 To compare the controller obtained with density optimization and the CBF controller, we pick one initial condition ${\begin{bmatrix}13 & 13 & 25\end{bmatrix}}^\intercal$ and vary $\alpha$ in the CBF implementation.

Fig. \ref{fig:CBF_comparison}(a) shows one simulation run with the optimal safe controller and the safety constraint is satisfied under the worst-case disturbance. Fig. \ref{fig:CBF_comparison}(b) shows the induced cost of CBF with different values of $\alpha$, and the cost associated with the optimal safe controller is lower than all of them. Fig. \ref{fig:CBF_comparison}(c) further shows the cost with different initial conditions, and the optimal safe controller clearly outperforms the CBF. In problems where an analytical and exact CBF is not known, one needs to use numerical methods to get an CBF, which is inevitably conservative. In those cases, the performance gap is expected to be even larger.
\section{Conclusion}\label{sec:conclusion}
This paper propose a density function approach for safe control synthesis. The approach utilizes the duality between density function and value function and constructs a primal-dual algorithm that solves the constrained optimal problem. By solving the worst-case disturbance as an optimal control problem, robust safety is guaranteed. When applied to the design of optimal safe control synthesis, since the proposed approach optimize the cumulative intervention, the obtained controller outperforms myopic controller such as the CBF controller.

One issue with the proposed approach is the computation complexity, which is dominated by the complexity of the HJB PDE. Possible solutions to this issue may include low-complexity approximation and parametrization of the value function and density function.

\balance
\bibliographystyle{abbrv}
\bibliography{density_bib}
\end{document}